\documentclass[11pt]{amsart}
\usepackage{epsfig}
\usepackage{graphics}
\usepackage{color}

\newtheorem{theorem}{Theorem}[section]

\newtheorem{lemma}[theorem]{Lemma}
\newtheorem{proposition}[theorem]{Proposition}
\newtheorem{corollary}[theorem]{Corollary}
\newtheorem{question}[theorem]{Question}

\newtheorem*{main thm}{Main Theorem}

\theoremstyle{definition}

\newtheorem{remark}[theorem]{Remark}

\theoremstyle{remark}

\def\mod{{\rm Mod}}

\def\bS {\Sigma_{g,n}}

\begin{document}

\newenvironment{prooff}{\medskip \par \noindent {\it Proof}\ }{\hfill
$\square$ \medskip \par}
    \def\sqr#1#2{{\vcenter{\hrule height.#2pt
        \hbox{\vrule width.#2pt height#1pt \kern#1pt
            \vrule width.#2pt}\hrule height.#2pt}}}
    \def\square{\mathchoice\sqr67\sqr67\sqr{2.1}6\sqr{1.5}6}
\def\pf#1{\medskip \par \noindent {\it #1.}\ }
\def\endpf{\hfill $\square$ \medskip \par}
\def\demo#1{\medskip \par \noindent {\it #1.}\ }
\def\enddemo{\medskip \par}
\def\qed{~\hfill$\square$}

\title[Long factorizations in mapping class groups]
{Arbitrarily long factorizations in mapping class groups}

\author[E. Dalyan, M. Korkmaz and M. Pamuk ]
{El\.{i}f Dalyan, Mustafa Korkmaz and Mehmetc\.{i}k Pamuk }

\address{(E.D.) Department of Mathematics, Hitit University,
 \c{C}orum, Turkey}
 \address{(M.K., M.P.) Department of Mathematics, Middle East Technical University, 06800
 Ankara, Turkey}
\email{elifdalyan@hitit.edu.tr} \email{korkmaz@metu.edu.tr} \email{mpamuk@metu.edu.tr}
\date{\today}

\begin{abstract}

On a compact oriented surface of genus $g$ with $n\geq 1$ boundary components,
$\delta_1, \delta_2,\ldots, \delta_n$, we consider positive factorizations of
the boundary multitwist $t_{\delta_1} t_{\delta_2} \cdots t_{\delta_n}$,
where $t_{\delta_i}$ is the positive Dehn twist about the boundary
$\delta_i$. We prove that for $g\geq 3$, the boundary multitwist $t_{\delta_1} t_{\delta_2}$
can be written as a product of arbitrarily large number of positive Dehn twists about
nonseparating simple closed curves, extending a recent result
of Baykur and Van Horn-Morris, who proved this result for $g\geq 8$.
This fact has immediate corollaries on the Euler characteristics of
the Stein fillings of contact three manifolds.

\end{abstract}

\maketitle
\setcounter{secnumdepth}{2}
\setcounter{section}{0}

\section{Introduction}
Let $\bS$ denote a compact connected oriented surface of genus $g$ with $n\geq 1$
boundary components, say $\delta_1, \delta_2,\ldots, \delta_n$, and let $\mod (\bS)$ denote the
mapping class group of $\Sigma_{g,n}$, the group of isotopy classes of self-diffeomorphisms
of $\Sigma_{g,n}$ fixing all points on the boundary. The study of the mapping class
group elements has important applications in low dimensional topology:
By the results of Giroux \cite{giroux} and Thurston-Winkelnkemper
\cite{{thurston-winkelnkemper}}, every open book decomposition
$(\bS, \Phi)$, where $\Phi \in \mod (\bS)$, of a closed oriented $3$-manifold $M$
admits a compatible contact structure and all contact structures on compact $3$-manifolds
come from open book decompositions. If the monodromy $\Phi$ can be written as a product of
positive Dehn twists, then the contact structure is Stein fillable. Writing $\Phi$ as a product
of positive Dehn twists provides a Stein filling of the contact $3$-manifold $M$ via Lefschetz
fibrations.

In the present paper we address the following question:
\emph{Is there a positive integer $N$ such that whenever the boundary multitwist
$t_{\delta_1} t_{\delta_2} \cdots t_{\delta_n}$ is written as a product of $k$ positive
nonseparating Dehn twists in $\mod (\bS)$, the integer $k$ must satisfy $k\leq N$?}
 Since any such factorization describes a Lefschetz fibration with $n$
disjoint sections of self-intersection $-1$, this question is equivalent to understanding whether or not
there is an upper bound on the number of singular fibers of Lefschetz fibrations admitting sections of
self-intersection $-1$?  Note that this is Question $2.3$ in \cite{auroux-problems}.

Recently, the above question is considered by Baykur and Van Horn-Morris~\cite{baykur-morris}:
They proved that for $g\geq 8$,
the boundary multitwist $t_{\delta_1} t_{\delta_2}$ in $\mod (\Sigma_{g,2})$ can be written as a product of arbitrarily large number
of positive Dehn twists about nonseparating simple closed curves.

In the present paper, we prove that the same conclusion can be drawn for all
$g\geq 3$.  For $g=2$, one needs an extra factor. The main result of our paper is the following theorem.

\begin{main thm}\label{main thm}
Let $a$ be a nonseparating simple closed curve on the genus--$g$ surface $\Sigma_{g,2}$ with two boundary
components $\delta_1$ and $\delta_2$.  In the mapping class group $\mod(\Sigma_{g,2})$, the multitwist
\begin{itemize}
  \item  [(i)]~ $t_{\delta_1} t_{\delta_2} t_a$ for $g=2$,
  \item  [(ii)]~ $t_{\delta_1} t_{\delta_2}$   for $g\geq 3$
\end{itemize}
can be written as a product of arbitrarily large number of positive Dehn twists about nonseparating
simple closed curves.
\end{main thm}

\begin{remark} \label{rmk:rm1}
It is then easy to conclude that in the mapping class group $\mod(\Sigma_{2,2})$,
the multitwist $t_{\delta_1}^2t_{\delta_2}^2$
can be written as a product of arbitrarily large number of positive Dehn twists about nonseparating
simple closed curves (see proof of Corollary~\ref{cor:1}).
It follows easily that if $g\geq 3$ and $k\geq 1$ then
$t_{\delta_1}^kt_{\delta_2}^k$ can be written as a product of arbitrarily large number of positive Dehn twists. The same holds for $g=2$ and $k\geq 2$.
\end{remark}

By capping off one of the boundary components, we obtain the following immediate corollary for surfaces
with one boundary component.

\begin{corollary} \label{cor:0.5}
Let $\Sigma_{g,1}$ be a compact connected oriented surface of genus $g$ with one boundary
component $\delta$. In the mapping class group $\mod(\Sigma_{g,1})$, the element
\begin{itemize}
  \item  [(i)]~ $t_{\delta}^2$ for $g=2$,
  \item  [(ii)]~ $t_{\delta}$   for $g\geq 3$
\end{itemize}
can be written as a product of arbitrarily large number of positive Dehn twists about nonseparating
simple closed curves.
\end{corollary}

In the mapping class group $\mod (\Sigma_g)$ of a closed orientable surface $\Sigma_g$
the identity element can be written as a product of positive
Dehn twists about nonseparating simple closed curves. It follows that
every element in $\mod (\Sigma_g)$ can be expressed as a product of
arbitrarily large number of nonseparating positive Dehn twists.
However, when $n\geq 1$, the situation is different; the identity element of $\mod (\Sigma_{g,n})$
admits no nontrivial factorization into a product of positive Dehn twists.

For $k\geq 1$, a factorization of the multitwist $t_{\delta_1}^k t_{\delta_2}^k\cdots t_{\delta_n}^k$
into a product of positive Dehn twists of the form
 \begin{equation} \label{eqn:1}
     t_{\delta_1}^k t_{\delta_2}^k\cdots t_{\delta_n}^k = \prod_{i=1}^r t_i
 \end{equation}
in the group $\mod(\Sigma_{g,n})$ describes a genus-$g$ Lefschetz fibration
$ X_g(r) \to S^2$ with $n$ disjoint sections such that the
self-intersection of each section is $-k$.  The Euler characteristic
of the total space $X_g(r)$ is given by
 \[
     \chi(X_g(r))= 2(2-2g) + r.
 \]

The next corollary, which is an improvement of the first part
of~\cite[Theorem 1.2]{baykur-morris}, follows from Remark~\ref{rmk:rm1}.

\begin{corollary}\label{cor:1}
For every $g\geq 3$, there is a family of genus-$g$ Lefschetz fibrations $X_g(r) \to S^2$
with two disjoint sections of self-intersection $-1$ such that
the set
$   \{ \chi(X_g(r)) \} $
of Euler characteristics is unbounded.
The same conclusion holds true for genus--$2$ Lefschetz fibrations but this
time with two disjoint sections of self-intersection $-2$.
\end{corollary}

Given a genus--$g$ Lefschetz fibration $f:X \to S^2$ with a section $\sigma$ and with a regular
fiber $\Sigma$, the complement of a regular neighborhood of the union $\Sigma\cup \sigma$ is
a Stein filling of its boundary $M$ equipped with the induced tight contact structure
(cf.~\cite{ozbagci-stipsicz}).
It was conjectured in~\cite{ozbagci-stipsicz} that the set
\[
 \mathcal{C}_{(M, \xi)}= \{ \chi(X) \ | \ X \  \textrm{is a Stein filling of} \ (M, \xi)\}
\]
is finite. In~\cite{baykur-morris0} and~\cite{baykur-morris}, it was shown that this conjecture is false.
Our main result provides more counter examples to this conjecture.
Once we have our Main Theorem,
then this follows from~\cite{ozbagci-stipsicz}, Theorem~3.2. See also refrences therein,
e.g.~\cite{akbulut-ozbagci}, \cite{etnyre-honda} and \cite{loi-piergallini}.

\begin{corollary}\label{cor:2}
For every $g\geq 2$, there is a contact $3$-manifold $(M_g, \xi_g)$, $M_g\neq M_{g'}$ for $g\neq g'$, admitting
infinitely many pairwise non-diffeomorphic Stein fillings such that
the set $\mathcal{C}_{(M_g, \xi_g)}$ is unbounded.
\end{corollary}

The manifold $M_g$ is the boundary of a regular neighborhood of the union $\Sigma_g\cup \sigma_g$
of the regular fiber $\Sigma_g$ and the section $\sigma_g$ of the Lefschetz fibration $X_g(r) \to S^2$
provided by Corollary~\ref{cor:1}.
The first integral homology of $M_g$ is a free abelian group of rank $2g$. Thus,
$M_g$ and $M_{g'}$ are not diffeomorphic for $g\neq g'$.

\begin{remark}
Plamenevskaya~\cite{plamenevskaya} and Kaloti~\cite{kaloti} showed that
if a contact $3$-manifold $(M, \xi)$ can be supported by a planar open book, then  $\mathcal{C}_{(M, \xi)}$
must be finite.  Hence, the contact structure supported by the open book with monodromy
$t_{\delta_1} t_{\delta_2} $, $g\geq 3$, cannot be supported by a planar open book.
\end{remark}

The organization of this paper is as follows:  In Section $2$, we obtain some preliminary results.
In Section $3$, we prove our main theorem.
In the last section we extend our results to surfaces with more boundary components.

\vspace{0.2in} \noindent{\bf {Acknowledgements.}} We started working on the problem after
\.{I}nan\c{c} Baykur  gave a talk on \cite{baykur-morris} at METU in December 2012.  We are grateful to him for
his explanations and helpful comments on this paper. After writing the first draft of the paper, we were informed by
Naoyuki Monden that he also obtained part of our results by similar arguments.  We also thank Burak Ozbagci for
very helpful comments on  earlier drafts of this paper.


\section{Preliminary Results}

Throughout this paper, we consider diffeomorphisms and curves up to isotopy; if two diffeomorphisms (resp. curves)
$x$ and $y$ are isotopic, we say that $x$ is equal to $y$.
We always use the functional notation for the multiplication in $\mod (\bS)$; the product $f h$ means that
$h$ applied first.
We refer to \cite{korkmaz-survey} and \cite{ozbagci-stipsicz} for the basics on
mapping class groups of surfaces, Lefschetz fibrations and contact structures.

For a simple closed curve $a$ on an oriented surface $\bS$,
we denote by $t_a$ the right (or positive) Dehn twist about $a$.
The diffeomorphism $t_a$ is
obtained by cutting $\bS$ along $a$ and gluing back after twisting one of the sides to the right by $360$ degrees.
For simplicity, the Dehn twist about the simple closed curve $c_i$ is denoted by $t_i$.

We now state two preliminary lemmas, which will be used frequently.

\begin{lemma} \label{lem:chainnn}
Let $c_1,c_2, \ldots, c_{2h+1}$ be simple closed curves on an oriented surface such that $c_i$
intersects $c_{i+1}$ transversely at one point for $i=1, \ldots, 2h$, and
$c_i \cap c_j =\emptyset$ if $\mid i-j\mid > 1$. Let $a$ and $b$ be the boundary components of
a regular neighborhood of $c_1\cup c_2\cup  \cdots \cup c_{2h+1}$.
The following equality holds:
 \[
    (t_1 t_2 \cdots t_{2h+1})^{2h+2}= t_a t_b.
 \]
\end{lemma}

\begin{lemma} \label{lem:chain}
Let $c_1, \ldots, c_n$ be simple closed curves on an oriented surface such that $c_i$
intersects $c_{i+1}$ transversely at one point for $i=1, \ldots, n-1$, and
$c_i \cap c_j =\emptyset$ if $\mid i-j\mid > 1$.  The following equality holds:
 \[
    (t_1 t_2 \cdots t_n)^{n+1}= (t_1 t_2 \cdots t_{n-1})^n t_n\cdots t_2 t_1 t_1t_2 \cdots t_n.
 \]
\end{lemma}
\begin{proof}
The proof follows easily from the braid relations $t_i t_{i+1} t_i = t_{i+1} t_i t_{i+1}$ 
and the relations $t_it_j=t_jt_i$ for $\mid i-j\mid > 1$. 
\end{proof}

\begin{figure}[hbt]
\begin{center}
      \includegraphics[width=5cm]{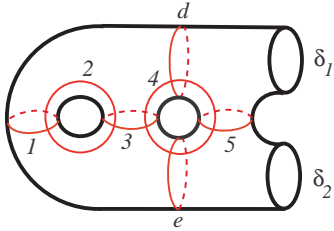}
      \caption{The curve labeled $i$ on $\Sigma_{2,2}$ represents $c_i$.}
      \label{fig:genus-2}
\end{center}
\end{figure}

Suppose that $g\geq 2$ and that the surface $\Sigma_{2,2}$
of genus $2$ illustrated in Figure~\ref{fig:genus-2}
is embedded in $\bS$.

\begin{lemma} \label{lem:diffeo}
Let $g\geq 2$ and let $x$ be a nonseparating simple closed curve on
$\bS$ intersecting $c_3$ and $d$ transversely only once.
In the mapping class group $\mod (\bS)$, the element
 \[
     \phi = t_4 t_3 t_2 t_1 t_1 t_2 t_3 t_4 t_x t_d t_3 t_x
 \]
maps the pair of simple closed curves $(c_3, d)$ to $(e, c_3)$.
\end{lemma}

\begin{proof}
It is easy to see that $t_x t_d t_3 t_x$ maps the pair $(c_3, d)$ to $(d, c_3)$ and
$t_4 t_3 t_2 t_1 t_1 t_2 t_3 t_4$ maps $(d, c_3)$ to $(e, c_3)$.
\end{proof}

Next we express a particular element of $\mod (\bS)$ as a product of arbitrarily large number
of positive Dehn twists.

\begin{proposition} \label{prop:tm}
Let $g\geq 2$ and let $x$ be a nonseparating simple closed curve on
$\bS$ intersecting $c_3$ and $d$ transversely only once.
Let $T= (t_1 t_2 t_3)^2 t_2 t_1 t_3 t_2$ in $\mod (\bS)$. For any positive integer $m$, we have
 \begin{equation*}
 \phi = t_4 t_3 t_2 t_1 t_1 t_2 t_3 t_4 t_x t_d t_3 t_x
  = t_3^{-m} t_e^mt_4 t_3 t_2 t_1 t_1 t_2 t_3 t_4 t_x t_d t_3 t_x t_e^{-m}t_3^m ~T^m.
 \end{equation*}
 In particular, for any positive integer $m$ the element
 $\phi$ may be written as a product of $12+10m$  positive Dehn twists
 about nonseparating simple closed curves.
\end{proposition}

\begin{proof}
By the two-holed torus relation and by applying the braid relation twice, we get
\begin{eqnarray*}
t_dt_e
      & = &(t_1 t_2 t_3)^4\\
      & =& (t_1 t_2 t_3)^2 t_1 t_2 t_1 t_3 t_2 t_3  \\
      &= & (t_1 t_2 t_3)^2 t_2 t_1 t_2 t_3 t_2 t_3 \\
      &= & (t_1 t_2 t_3)^2 t_2 t_1 t_3 t_2 t_3 t_3 . \\
\end{eqnarray*}
Hence, we have $(t_1 t_2 t_3)^2t_2 t_1 t_3 t_2 = t_d t_3^{-1} t_e t_3^{-1}$.
By taking the $m$th power of both sides, we obtain
$T^m = t_d^m t_3^{-m}t_e^m t_3^{-m}$ for any positive integer $m$. (Note that this equality is used
in~\cite{baykur-korkmaz-monden}.)
By using Lemma~\ref{lem:diffeo}, this leads to
\begin{eqnarray*}
t_4 t_3 t_2 t_1 t_1 t_2 t_3 t_4 t_x t_d t_3 t_x
 &=& (t_4 t_3 t_2 t_1 t_1 t_2 t_3 t_4 t_x t_d t_3 t_x )  t_d^{-m} t_3^{m} t_e^{-m} t_3^{m}T^m\\
 &=& \phi \, t_d^{-m} t_3^{m} t_e^{-m}    t_3^{m}T^m\\
 &=& t_{\phi(d)}^{-m} t_{\phi(3)}^{m} \phi \, t_e^{-m}    t_3^{m}T^m\\
 &=& (t_{3}^{-m} t_{e}^{m} \phi \, t_e^{-m} t_3^{m} ) T^m,
\end{eqnarray*}
where $\phi=t_4 t_3 t_2 t_1 t_1 t_2 t_3 t_4 t_x t_d t_3 t_x$.
Since the conjugate of $\phi$, $t_{3}^{-m} t_{e}^{m} \phi \, t_e^{-m}    t_3^{m}$ is a product of
$12$ positive Dehn twists and $T$ is a product of $10$ positive Dehn twists (all about nonseparating
curves), the proposition is proved.
\end{proof}

Proposition~\ref{prop:tm} gives a short counter example to two conjectures of
Ozbagci and Stipsicz (\cite{ozbagci-stipsicz2}, Conjecture 12.3.16 and Conjecture 15.3.5).

\begin{corollary} \label{cor:infty}
The contact $3$--manifold compatible with the open book whose monodromy is given by
$\phi$ in Proposition~\ref{prop:tm} has Stein fillings with arbitrarily large Euler characteristics.
\end{corollary}


\section{Surfaces with two boundary components}
We prove our main theorem in this section.
In the proof, our strategy will be to look for the factorization of
$\phi$ given in Lemma~\ref{lem:diffeo} in a
factorization of a given element.
Once we find $\phi$ in the factorization as a subword,
as in Proposition~\ref{prop:tm} it will give us a factorization consisting of
an arbitrarily large number of positive Dehn twists.

\subsection{Proof of the Main Theorem}
Suppose first that $g=2$.
Consider the surface $\Sigma_{2,2}$ of genus $2$ illustrated in Figure~\ref{fig:genus-2}.
By Lemma~\ref{lem:chainnn}, we may write the product $t_{\delta_1} t_{\delta_2}$ of the Dehn twists about the
boundary parallel curves $\delta_1$ and $\delta_2$ as a product
of $30$ positive Dehn twists as
\begin{equation}\label{eqn:two boundary}
t_{\delta_1} t_{\delta_2} = (t_1 t_2 t_3 t_4 t_5)^6.
\end{equation}

Using Lemma~\ref{lem:chain}, we can rewrite Equation~\ref{eqn:two boundary} as
\begin{eqnarray*}
 t_{\delta_1} t_{\delta_2}
   &=& (t_1 t_2 t_3 t_4 t_5)^6 \\
   &=& (t_1 t_2 t_3 t_4 )^5 t_5 t_4 t_3 t_2 t_1 t_1 t_2 t_3 t_4 t_5 \\
   &=& t_1 t_2 t_3 t_4 (t_1 t_2 t_3 t_4 )^4 t_5 t_4 t_3 t_2 t_1 t_1 t_2 t_3 t_4 t_5 \\
   &=& t_1 t_2 t_3 t_4  (t_1 t_2 t_3)^4 t_4 t_3 t_2 t_1 t_5 t_4 t_3 t_2 t_1 t_1 t_2 t_3 t_4 t_5 .
\end{eqnarray*}
Now applying the two--holed torus relation $(t_1 t_2 t_3)^4=t_d t_e$,
a special case of Lemma~\ref{lem:chainnn}, we write
\begin{eqnarray} \label{eqn:3-1}
    t_{\delta_1} t_{\delta_2}
      &=& t_1 t_2 t_3 t_4 t_d t_e t_4 t_3 t_2 t_1 t_5 t_4 t_3 t_2 t_1 t_1 t_2 t_3 t_4 t_5 \nonumber \\
      &=& t_5 t_1 t_2 t_3 t_4 t_d t_e t_4 t_3 t_2 t_1 t_5 (t_4 t_3 t_2 t_1 t_1 t_2 t_3 t_4).
\end{eqnarray}
Reordering the terms by conjugation, we may write
 \begin{equation}\label{eqn:d_9}
    t_{\delta_1} t_{\delta_2}= D_9 (t_4 t_3 t_2 t_1 t_1 t_2 t_3 t_4) t_4 t_d t_3,
  \end{equation}
where $D_9$ is a product of nine positive Dehn twists. More precisely,
\begin{eqnarray*}
D_9
  &=&  (f t_5t_1t_2t_3f^{-1})  (h t_et_4h^{-1}) t_2t_1t_5\\
  &=&  t_{f(c_5)}  \,  t_{f(c_1)} \,  t_{f(c_2)} \,
    t_{f(c_3)}  \, t_{h(e)} \,  t_{h(c_4)}  \,  t_2t_1t_5
\end{eqnarray*}
for $f=(t_4t_dt_3)^{-1}$ and $h=t_3^{-1}$. Here we use the fact that the Dehn twists
$t_{\delta_i}$ are central in the mapping class group.

Multiplying both sides of the equality~(\ref{eqn:d_9}) with $t_4$  gives
\[
   t_{\delta_1} t_{\delta_2}  t_4 = D_9 t_4 t_3 t_2 t_1 t_1 t_2 t_3 t_4 t_4 t_d t_3 t_4 .
\]
Since $c_4$ is a simple closed curve intersecting $c_3$ and $d$ transversely once,
it follows now from Proposition~\ref{prop:tm} that $t_{\delta_1} t_{\delta_2}t_4$
may be written as a product of $21+10m$ positive Dehn twists for any integer $m\geq 1$.

The Dehn twists
$t_{\delta_1}$ and $ t_{\delta_2}$ are central in $\mod(\Sigma_{2,2})$. Since the curve $c_4$
is nonseparating and since any two Dehn twists about nonseparating curves
are conjugate, it follows that
$t_{\delta_1} t_{\delta_2}t_a$
may be written as a product of $21+10m$ positive Dehn twists.

This proves part~(i) of the Main Theorem.

\begin{figure}[hbt]
\begin{center}
      \includegraphics[width=8cm,height=4cm]{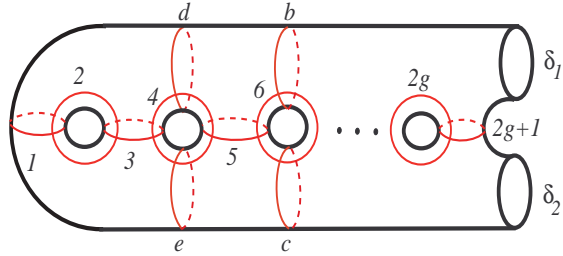}
      \caption{The curve labeled $i$ on $\Sigma_{g,2}$ represents $c_i$.}
      \label{fig:genus-g}
\end{center}
\end{figure}

Suppose now that $g\geq 3$ and that  $\Sigma_{g,2}$ is the surface in Figure~\ref{fig:genus-g}.
In the group $\mod(\Sigma_{g,2})$, using Lemma~\ref{lem:chainnn} we may write
\begin{eqnarray*}
   t_{\delta_1} t_{\delta_2}
     &=& (t_1 t_2 \cdots t_{2g+1})^{2g+2} \\
     &=& (t_1 t_2 t_3 t_4 t_5)^6 t_1 K,
\end{eqnarray*}
where $K$ is a product of positive Dehn twists. In fact, $K$ is a product of
$2(g+1)(2g+1)-31$ positive Dehn twists all are about nonseparating curves.
Since $(t_1 t_2 t_3 t_4 t_5)^6=t_b t_c$, we may write
\[
   t_{\delta_1} t_{\delta_2} = t_b t_c t_1 K .
\]
By part (i), the right--hand side of this equality may be written as a product
of $2(g+1)(2g+1)-10+10m$ positive Dehn twists for any positive integer $m$.

This proves (ii), finishing the proof of the theorem.
\qed

\begin{remark}\label{min gen}
We want to emphasize that $g=3$ is the smallest possible genus
such that the the product of boundary parallel Dehn twists may be written as a product of
arbitrarily large number of positive Dehn twists in the mapping class group of $\Sigma_{g,2}$.
In the case of genus $2$, by capping off one of the boundary components of $\Sigma_{2,2}$, one can immediately
see that in $\mod(\Sigma_{2,1})$, for any nonseparating simple closed curve $a$,
$t_\delta t_a$ can be written as a product of arbitrarily large number
of positive Dehn twists about nonseparating curves.  Now we want to point out that one cannot get
rid of $t_a$ in the above factorization.  Writing $t_\delta$ as a product of
arbitrarily large number of positive Dehn twists would imply the existence of
genus--$2$ Lefschetz fibrations over $S^2$ with section of square $-1$ and with arbitrarily large
Euler characteristic, contradicting~\cite[Theorem~1.4]{smith}.
\end{remark}

\subsection{Proof of Corollary \ref{cor:1}}
We know that the product
 \[
     t_{\delta_1}^k t_{\delta_2}^k = \prod_{i=1}^r t_i
 \]
in $\mod(\Sigma_{g,2})$ describes a genus-$g$ Lefschetz fibration
$ X_g(r) \to S^2$ with two disjoint sections of self-intersection $-k$.
Moreover the Euler characteristic of $X_g(r)$ is given by
 \[
     \chi(X(r))= 2(2-2g) + r.
 \]

Let $g\geq 3$.
For each positive integer $m$, let $X_m$ be the total space of the Lefschetz fibration given by the
factorization of $t_{\delta_1} t_{\delta_2}$ into the product of $2(g+1)(2g+1)+10(m-1)$ Dehn twists.
The Euler characteristic of $X_m$ is
\[ \chi(X_m)= 2(2-2g) + 2(g+1)(2g+1)+10(m-1).\]

Let $g=2$. It is easy to see that
$t_{\delta_1}^2 t_{\delta_2}^2$ can be written as a product of arbitrarily large
number of positive Dehn twists about nonseparating curves: By Equation~(\ref{eqn:3-1}), we have
\begin{eqnarray}
    t_{\delta_1}^2 t_{\delta_2}^2
      &=& (t_1 t_2 t_3 t_4 t_d t_e t_4 t_3 t_2 t_1 t_5 t_4 t_3 t_2 t_1 t_1 t_2 t_3 t_4 t_5)^2.
\end{eqnarray}
It is then easy to see that we may write
\begin{eqnarray}
    t_{\delta_1}^2 t_{\delta_2}^2
      &=& t_4 t_3 t_2 t_1 t_1 t_2 t_3 t_4 t_4 t_d t_3 t_4 D'',
\end{eqnarray}
where $D''$ is a product of $28$ Dehn twists. Hence, by Proposition~\ref{prop:tm},
$t_{\delta_1}^2 t_{\delta_2}^2$ can be written as a product of $40+10m$ Dehn twists for
any $m$. The total space $X_m$ of the corresponding Lefschetz fibration
has the Euler characteristic
 \[ \chi(X_m)= 36+10m.
 \]
The self-intersection number of each of the two sections is $-2$.
\qed


\section{Surfaces with more boundary components}

Our aim in this section is to increase the number of boundary components and write certain elements in the mapping
class group as a product of arbitrarily large number of Dehn twists. Our main tool is the relations obtained in \cite{korkmaz-ozbagci}.

\begin{theorem}
Let $g\geq 3$ and $a$ be a nonseparating simple closed curve on the surface $\Sigma_{g,6}$
of genus $g$ with $6$ boundary components.
In the mapping class group $\mod(\Sigma_{g,6})$, the multitwist
\[
  t_{\delta_1} t_{\delta_2} t_{\delta_3} t_{\delta_4} t_{\delta_5} t_{\delta_6} t_a
\]
can be written as a product of arbitrarily large number of positive
Dehn twists about nonseparating simple closed curves.
\end{theorem}

\begin{figure}[hbt]
\begin{center}
      \includegraphics[width=12cm]{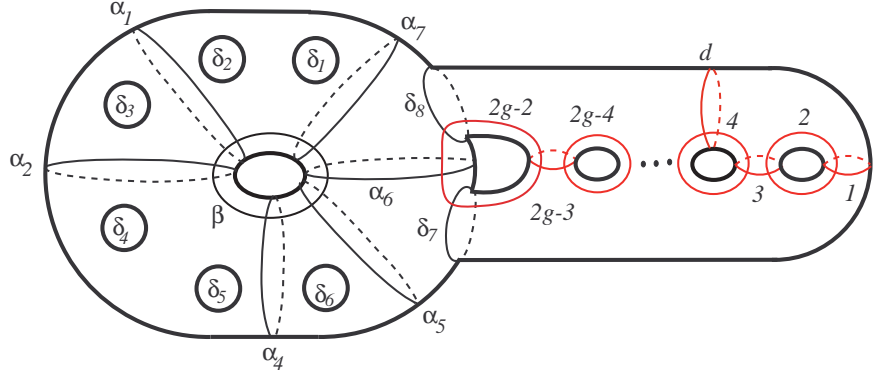}
      \caption{The curve labeled $i$ is $c_i$.}
      \label{fig:genus-1-bdry-8}
\end{center}
\end{figure}

\begin{proof}
Consider the eight-holed torus $\Sigma_{1,8}$ given in Figure~\ref{fig:genus-1-bdry-8}
obtained by cutting $\Sigma_{g,6}$ along $\delta_7\cup\delta_8$.
The eight--holed torus relation given in~\cite{korkmaz-ozbagci} says
 \begin{equation}\label{eqn:8-holed}
   t_{\delta_1} t_{\delta_2} \cdots t_{\delta_8}= t_{\alpha_4} t_{\alpha_5} t_{\beta_1} t_{\sigma_3} t_{\sigma_6}
   t_{\alpha_2} t_{\beta_6} t_{\sigma_4} t_{\sigma_7} t_{\alpha_7} t_{\beta_4} t_{\sigma_5},
\end{equation}
where $t_{\beta_1} = t_{\alpha_1} t_{\beta} t_{\alpha_1}^{-1}$, \ $t_{\beta_4} = t_{\alpha_4} t_{\beta} t_{\alpha_4}^{-1}$ \ and \
$t_{\beta_6} = t_{\alpha_6} t_{\beta} t_{\alpha_6}^{-1}$. Here, we wrote $\delta_{i+1}$ for the curve
$\delta_{i}$ of Figure~$8$ in~\cite{korkmaz-ozbagci}, $\delta_{1}$ for $\delta_{8}$, and the nonseparating curves
$\sigma_i$ on $\Sigma_{1,8}$ are not drawn.

By conjugating with appropriate elements, we may rewrite the equality~(\ref{eqn:8-holed}) as
 \begin{equation}\label{eqn:eqn3}
   t_{\delta_1} t_{\delta_2} \cdots t_{\delta_8}= t_{\alpha_5} t_{\alpha_7} D ,
\end{equation}
where $D$ is a product of $10$ positive Dehn twists.

By letting $\alpha_6=c_{2g-1}$ and by using Lemmas~\ref{lem:chainnn} and~\ref{lem:chain},
we write
\begin{eqnarray*}
 t_{\alpha_5} t_{\alpha_7}
    &=& (t_1 t_2\cdots t_{2g-1})^{2g} \\
    &=&(t_1  t_2\cdots t_{2g-2})^{2g-1} (t_{2g-1} \cdots  t_1 t_1 \cdots t_{2g-1})\\
    &=&(t_1  t_2\cdots t_{2g-3})^{2g-2}
    (t_{2g-2} \cdots   t_1 t_1  \cdots t_{2g-2}) (t_{2g-1} \cdots  t_1 t_1  \cdots t_{2g-1}),
\end{eqnarray*}
where $t_i$ is the Dehn twist about the curve $c_i$. By using Lemma~\ref{lem:chainnn} again,
we get the equality
\begin{equation} \label{eqn:eqn4}
 t_{\alpha_5} t_{\alpha_7}
    = t_{\delta_7} t_{\delta_8}   (t_{2g-2} \cdots   t_1 t_1  \cdots t_{2g-2}) (t_{2g-1} \cdots  t_1 t_1  \cdots t_{2g-1}).
\end{equation}

Now replace the product $t_{\alpha_5} t_{\alpha_7}$ in~(\ref{eqn:eqn3}) with
the right--hand side of~(\ref{eqn:eqn4})
and cancel $t_{\delta_7} t_{\delta_8}$ from both sides. The result is
\begin{equation*}
 t_{\delta_1} t_{\delta_2} \cdots t_{\delta_6}
   = (t_{2g-2} \cdots   t_2 t_1 t_1  t_2 \cdots t_{2g-2}) (t_{2g-1} \cdots  t_2 t_1 t_1  t_2 \cdots t_{2g-1})D
\end{equation*}
which may be rewritten as
\begin{equation}\label{eqn:mmmm}
 t_{\delta_1} t_{\delta_2} \cdots t_{\delta_6}
   = t_4 t_3t_2 t_1 t_1  t_2t_3 t_4t_4 t_3t_4 D'.
\end{equation}
where $D'$ is a product of $8g-7$ positive Dehn twists about nonseparating curves.
Since the each Dehn twist on the right--hand side
of~(\ref{eqn:mmmm}) commutes with the Dehn twists on the left--hand side, by multiplying both sides
by $t_d$ we may write
\begin{equation}\label{eqn:mmmx}
 t_{\delta_1} t_{\delta_2} \cdots t_{\delta_6}t_d
   = t_4 t_3t_2 t_1 t_1  t_2t_3 t_4t_4 t_dt_3t_4 D'=\phi D',
\end{equation}
where $\phi=t_4 t_3t_2 t_1 t_1  t_2t_3 t_4t_4 t_dt_3t_4$.
It follows now from Proposition~\ref{prop:tm} that, for any positive integer $m$,
multitwist $t_{\delta_1} t_{\delta_2} \cdots t_{\delta_6}t_d$
may be written as a product of $8g+5+10m$ positive Dehn twists about
nonseparating simple closed curves.
\end{proof}

\begin{theorem}
    Let $a$ be a nonseparating simple closed curve on $\Sigma_{2,4}$
    of genus $2$ with $4$ boundary components.
    In the mapping class group $\mod(\Sigma_{2,4})$, the multitwist
    \[
    t_{\delta_1} t_{\delta_2} t_{\delta_3} t_{\delta_4} t_a
    \]
    can be written as a product of arbitrarily large number of positive
    Dehn twists about nonseparating simple closed curves.
\end{theorem}

\begin{proof} To get a certain relation in the mapping class group on a torus with six holes, we use
the technique of~\cite{korkmaz-ozbagci}.
We start with the four-holed torus relation of \cite{korkmaz-ozbagci} (see Figure \ref{fig:genus-1-bdry-5}~(a)
for the curves):
\begin{equation}\label{eqn:torus4delik}
    t_{\delta_1} t_{\delta_2} t_{\delta_3} t_{\gamma} =
    (t_{\beta} t_{\alpha_1} t_{\alpha_3} t_{\beta} t_{\alpha_2}  t_{\alpha_5})^2.
\end{equation}

\begin{figure}[hbt]
\begin{center}
      \includegraphics[width=13cm]{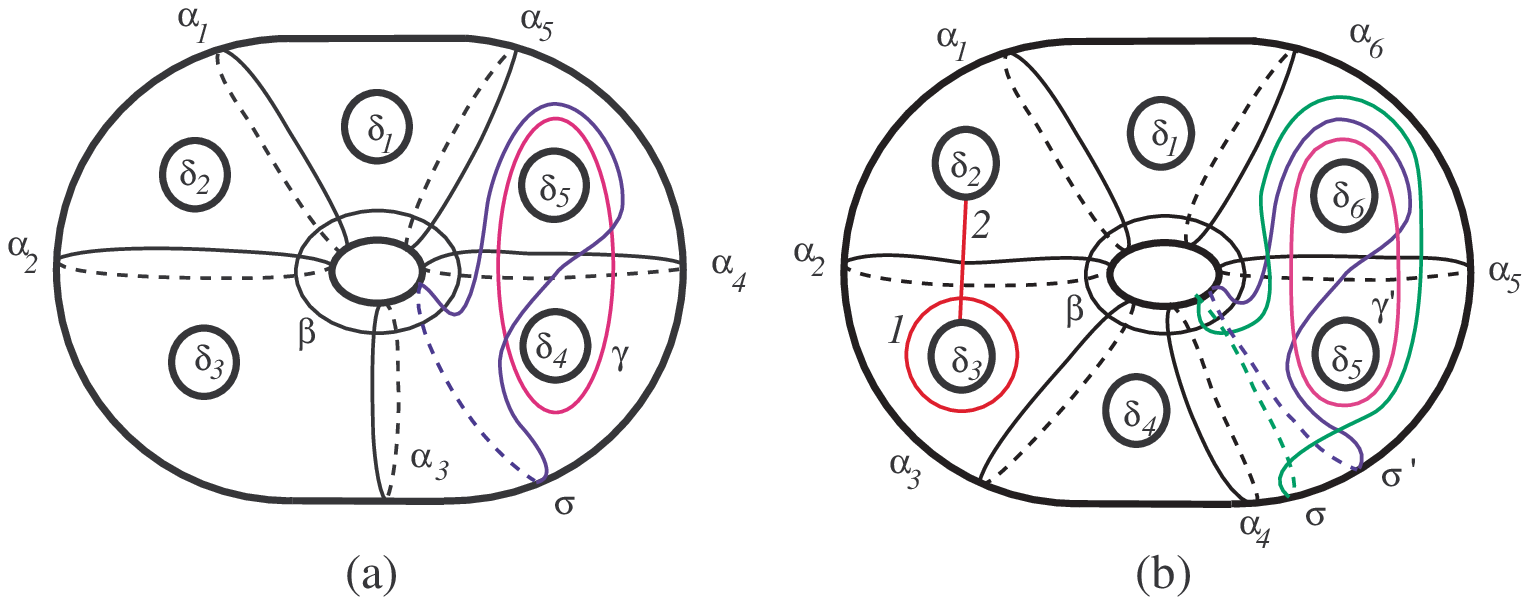}
      \caption{ }
      \label{fig:genus-1-bdry-5}
\end{center}
\end{figure}

Multiplying both sides of~(\ref{eqn:torus4delik}) with $t_{\sigma} t_{\alpha_4}$, we get
\begin{equation}\label{eqn:xyy}
 t_{\delta_1} t_{\delta_2} t_{\delta_3} t_{\gamma} t_{\sigma} t_{\alpha_4} =
 t_{\beta} t_{\alpha_1} t_{\alpha_3} t_{\beta} t_{\alpha_2} t_{\alpha_5} t_{\beta} t_{\alpha_1}
 t_{\alpha_3} t_{\beta} t_{\alpha_2} t_{\alpha_5} t_{\sigma} t_{\alpha_4}.
 \end{equation}
Using the lantern relation $t_{\gamma} t_{\sigma} t_{\alpha_4} =t_{\delta_4} t_{\delta_5} t_{\alpha_3} t_{\alpha_5}$,
from Equation~(\ref{eqn:xyy}) we get
 \begin{eqnarray*}
    t_{\delta_1} t_{\delta_2} t_{\delta_3} t_{\delta_4} t_{\delta_5} t_{\alpha_3} t_{\alpha_5}
     &=& t_{\beta} t_{\alpha_1} t_{\alpha_3} t_{\beta} t_{\alpha_2} t_{\alpha_5} t_{\beta} t_{\alpha_1}
      t_{\alpha_3} t_{\beta} t_{\alpha_2} t_{\alpha_5} t_{\sigma} t_{\alpha_4}\\
     &=&t_{\alpha_3} t_{\beta_3} t_{\alpha_1}  t_{\beta} t_{\alpha_2} t_{\alpha_5} t_{\beta} t_{\alpha_1} t_{\alpha_3}
      t_{\beta} t_{\alpha_2} t_{\sigma} t_{\alpha_4}  t_{\alpha_5},
 \end{eqnarray*}
where $\beta_3= t_{\alpha_3}^{-1}(\beta)$. Canceling  $t_{\alpha_3} t_{\alpha_5}$ from both sides gives us
 \begin{equation}\label{eqn:xxx}
    t_{\delta_1} t_{\delta_2} t_{\delta_3} t_{\delta_4} t_{\delta_5}=
     t_{\beta_{{3}}} t_{\alpha_1} t_{\beta} t_{\alpha_2}
     t_{\alpha_5} t_{\beta} t_{\alpha_1} t_{\alpha_3} t_{\beta} t_{\alpha_2} t_{\sigma}  t_{\alpha_4}.
  \end{equation}

Next we are going to repeat the same procedure for the six-holed torus $\Sigma_{1,6}$
(see Figure \ref{fig:genus-1-bdry-5}~(b)). By the equality~(\ref{eqn:xxx}), we have
\[
   t_{\delta_1} t_{\delta_2} t_{\delta_3} t_{\delta_4} t_{\gamma'}
   =t_{\beta_{{3}}} t_{\alpha_1} t_{\beta} t_{\alpha_2}
   t_{\alpha_6} t_{\beta} t_{\alpha_1} t_{\alpha_3} t_{\beta} t_{\alpha_2} t_{\sigma}  t_{\alpha_4}.
\]
Multiplying both sides with $t_{\sigma'} t_{\alpha_5}$ and using the lantern relation
  \[
  t_{\gamma'} t_{\sigma'} t_{\alpha_5}= t_{\delta_5}
  t_{\delta_6} t_{\alpha_4} t_{\alpha_6}
  \]
 give
\begin{align*}
    t_{\delta_1} t_{\delta_2} t_{\delta_3} t_{\delta_4} t_{\delta_5}
    t_{\delta_6} t_{\alpha_4} t_{\alpha_6}
  &=
    t_{\beta_{{3}}} t_{\alpha_1} t_{\beta} t_{\alpha_2}
    t_{\alpha_6} t_{\beta} t_{\alpha_1} t_{\alpha_3} t_{\beta} t_{\alpha_2}
    t_{\sigma}  t_{\alpha_4} t_{\sigma'} t_{\alpha_5}\\
  &=
    t_{\alpha_6} t_{\beta_{{36}}} t_{\alpha_1} t_{\beta_{6}}
    t_{\alpha_2} t_{\beta} t_{\alpha_1} t_{\alpha_3} t_{\beta}
    t_{\alpha_2} t_{\sigma} t_{\sigma'} t_{\alpha_5} t_{\alpha_4},
\end{align*}
where $\beta_{36}= t_{\alpha_6}^{-1} (\beta_3)$ \
and \ $\beta_{6}= t_{\alpha_6}^{-1} (\beta)$.  Cancel $ t_{\alpha_4}$
and $ t_{\alpha_6}$ from both sides of the equality to get
\begin{equation*}
  t_{\delta_1} t_{\delta_2} t_{\delta_3} t_{\delta_4} t_{\delta_5} t_{\delta_6}
  = t_{\beta_{{3 6}}} t_{\alpha_1} t_{\beta_{6}}
  t_{\alpha_2} t_{\beta} t_{\alpha_1} t_{\alpha_3} t_{\beta}
  t_{\alpha_2} t_{\sigma} t_{\sigma'} t_{\alpha_5},
\end{equation*}
or
\begin{equation}\label{eqn:zzz}
    t_{\delta_1} t_{\delta_2} t_{\delta_3} t_{\delta_4} t_{\delta_5} t_{\delta_6}
  = t_{\alpha_1} t_{\alpha_3}t_{\beta} t_{\alpha_2} t_{\sigma} t_{\sigma'} t_{\alpha_5}
    t_{\beta_{{3 6}}} t_{\alpha_1} t_{\beta_{6}}  t_{\alpha_2} t_{\beta}.
\end{equation}

Now we glue $\delta_2$ to $\delta_3$ to get a surface $\Sigma_{2,4}$ of genus two with
four boundary components. We glue $\delta_2$ to $\delta_3$ so that the curve $c_2$ is closed.
Let us set $c_1=\delta_3$, $c_3=\alpha_2$, $c_4=\beta$ and $\alpha_1=d$. Write
 \[
 t_{\alpha_1} t_{\alpha_3}=(t_1t_2t_3)^4=t_1t_1t_2t_1t_3t_2(t_1t_2t_3)^2
 \]
in the equality~(\ref{eqn:zzz}) and cancel $t_1^2$ from both sides. The result can be written as
\begin{eqnarray}\label{eqn:zcc}
    t_{\delta_1} t_{\delta_4} t_{\delta_5} t_{\delta_6}
  &=& t_2t_1t_3t_2(t_1t_2t_3)^2 t_4 t_3 t_{\sigma} t_{\sigma'} t_{\alpha_5}
    t_{\beta_{{3 6}}} t_d t_{\beta_{6}}  t_3 t_4 \nonumber \\
  &=&  t_2t_1t_4t_3t_2t_1t_2t_3t_1t_2t_3 t_4
     t_3 t_{\sigma} t_{\sigma'} t_{\alpha_5}
     t_{\beta_{36}} t_d t_{\beta_{6}}  t_3.
\end{eqnarray}
Now multiply both sides of the equality~(\ref{eqn:zcc}) with $t_{\beta_{36}}$ and let
 \[
     \phi= t_4 t_3 t_2 t_1 t_1 t_2 t_3 t_4
     t_{\beta_{36}} t_d t_3 t_{\beta_{36}}.
 \]
It is clear that one can write
\[
t_{\delta_1} t_{\delta_4} t_{\delta_5} t_{\delta_6} t_{\beta_{36}}=\phi D,
\]
where $D$ is a product of $9$ positive Dehn twists all about nonseparating curves.
We now apply Proposition~\ref{prop:tm} to conclude that $t_{\delta_1} t_{\delta_4} t_{\delta_5} t_{\delta_6} t_{\beta_{36}}$
is a product of $10m+21$ Dehn twists.

This concludes the proof of the theorem.
\end{proof}

\section{A question}
The results of this paper prompt the following question:

\begin{question}
Let $g\geq 2$ and let $n\geq 3$. In the mapping class group of the oriented surface
$\Sigma_{g,n}$, is it possible to write the multitwist $t_{\delta_1} t_{\delta_2} \cdots t_{\delta_n}$
as a product of positive Dehn twists about nonseparating simple closed curves? If yes,
can it be written as a product of arbitrarily large number of such Dehn twists?
\end{question}

Note that in the case $g=1$, the second author and Ozbagci~\cite{korkmaz-ozbagci}
proved that $t_{\delta_1} t_{\delta_2} \cdots t_{\delta_n}$
can be written as a product of positive Dehn twists about nonseparating simple closed curves
for $n\leq 9$, but it cannot be written if $n\geq 10$. Moreover, whenever
$t_{\delta_1} t_{\delta_2} \cdots t_{\delta_n}$ is a product of positive
Dehn twists about nonseparating simple closed curves, the number of such Dehn twists
must be $12$ for homological reasons.

In the case $g=2$,  Onaran~\cite{onaran} showed that,
in the mapping class group of $\Sigma_{2,n}$, the boundary multitwist $t_{\delta_1} t_{\delta_2} \cdots t_{\delta_n}$
can be written as a product of positive Dehn twists about nonseparating simple closed curves
for $n\leq 8$. In~\cite{tanaka}, Tanaka improved this result to $n=4g+4$ for any $g$.

\bibliographystyle{amsplain}
\providecommand{\bysame}{\leavevmode\hbox
to3em{\hrulefill}\thinspace}
\providecommand{\MR}{\relax\ifhmode\unskip\space\fi MR }
\providecommand{\MRhref}[2]{%
  \href{http://www.ams.org/mathscinet-getitem?mr=#1}{#2}
} \providecommand{\href}[2]{#2}

\end{document}